\documentclass[10pt]{article}

\usepackage{amsmath, amsthm, amsfonts}
\usepackage{amssymb}
\usepackage[mathscr]{eucal}
\newtheorem{thm}{Theorem}
\newtheorem{lemma}{Lemma}

\newtheorem{prop}{Proposition}
\usepackage{graphicx}
\def\C#1{{\mathcal {#1}}}

\title
{Convergence of Rothe scheme for hemivariational inequalities of parabolic type}

\author{Piotr Kalita,\\
  Faculty of Mathematics and Computer Science,\\
  Institute of Computer Science,\\
  Jagiellonian University,\\
  ul. prof. S. \L{}ojasiewicza 6, 30-348 Krak\'{o}w, Poland\\
  \texttt{piotr.kalita@uj.edu.pl}}

\begin{document}


\maketitle

\begin{abstract}This article presents the convergence analysis of a sequence of piecewise constant and
piecewise linear functions obtained by the Rothe method to the solution of the first order
evolution partial differential inclusion $u'(t)+Au(t)+\iota^*\partial J(\iota u(t))\ni f(t)$, where
the multivalued term is given by the Clarke subdifferential of a locally Lipschitz functional. The
method provides the proof of existence of solutions alternative to the ones known in literature and
together with any method for underlying elliptic problem, can serve as the effective tool to
approximate the solution numerically. Presented approach puts into the unified framework known
results for multivalued nonmonotone source term and boundary conditions, and generalizes them to
the case where the multivalued term is defined on the arbitrary reflexive Banach space as long as
appropriate conditions are satisfied. In addition the results on improved convergence as well as
the numerical examples are presented.
\end{abstract}

\section{Introduction}
Partial differential inclusions with the multivalued term given in the form of Clarke
subdifferential are known as hemivariational inequalities (HVIs). HVIs are the natural
generalization of the inclusions with monotone multivalued term (which lead to variational
inequalities) and were firstly considered by Panagiotopoulos in early 1980s. For the description of
the origins of HVIs and underlying mathematical theory we refer the reader to the book
\cite{Naniewicz1995}.

\noindent This paper deals with the first order evolution inclusion of type
$u'(t)+A(u(t))+\iota^*\partial J(\iota u(t)) \ni f(t)$. Such problems are known as parabolic HVIs
or boundary parabolic HVIs depending whether an operator $\iota$ is the embedding operator from
$H^1(\Omega)$ to $L^2(\Omega)$ or the trace operator from $H^1(\Omega)$ to
$H^{\frac{1}{2}}(\partial\Omega)$. The first case corresponds to multivalued and nonmonotone source
term in the equation and the second one to multivalued and nonmonotone boundary conditions of
Neumann-Robin type. Such inclusions  are used to model the diffusive transport through
semipermeable membranes where the multivalued term represents the semipermeability relation
\cite{Miettinen1999} and the temperature control problems where the multivalued term represents the
feedback control \cite{Haslinger1999}, \cite{Guanghui2010}.

\noindent The existence of solutions to problems governed by inclusions of considered type was
investigated by many authors. There are several techniques used to obtain the existence results:
\begin{itemize}
\item Classical Faedo-Galerkin approach combined with the regularization of the multivalued term by means of a standard mollifier; solutions of underlying system of ordinary differential equations are proved to converge (in appropriate sense) to the function which is shown to be the solution of analyzed HVI. This technique was used in context of parabolic HVIs by Miettinen \cite{Miettinen1996}, Miettinen and Panagiotopoulos \cite{Miettinen1999} and Goeleven et al. \cite{Goeleven2003}.
\item The approach based on the notion of upper and lower solutions. The solution is shown to be the limit of solutions of problems governed by the equations obtained by the regularization of the multivalued term together with the truncation by the lower and upper solutions. The distinctive feature of this approach is that the growth conditions on the multivalued term are replaced by the assumption of the existence of lower and upper solutions. The technique was used for parabolic HVIs by Carl \cite{Carl1996} and developed in \cite{Carl2002}, \cite{Carl2003}, \cite{Carl2004}, \cite{Carl2008}.
\item The technique based on showing that the analyzed HVI satisfies the assumptions of the general framework for which the appropriate surjectivity result holds. This approach was used by Liu \cite{Liu2000} and by Mig\'{o}rski \cite{Migorski2001} and developed for the boundary case in \cite{Migorski2004}.
\item The technique based on adding to the inclusion the regularizing term multiplied by $\epsilon>0$, showing that the solutions to obtained problems satisfy some bounds uniformly in $\epsilon$ and passing to the limit $\epsilon\to 0$. This technique was used for parabolic HVIs by Liu and Zhang \cite{Liu1998} and Liu \cite{Liu1999} and developed in \cite{Liu2003}, \cite{Liu2005}.
\end{itemize}
\noindent It should be remarked that above techniques are either nonconstructive (i.e. they are
based on surjectivity result) or constructive but not effective (i.e. require a priori knowledge of
lower and upper solutions, or require additional or smoothing terms in the problem).

\noindent In contrast to the existence theory, numerical methods to approximate effectively the
solutions to parabolic HVIs were not considered by many authors. In the book of Haslinger,
Miettinen and Panagiotopoulos \cite{Haslinger1999} the convergence of solutions obtained by the
finite element approximation of the space variable and finite difference approximation of the time
variable is proved. However only the case of the linear operator $A$ and the multivalued source
term (and not boundary conditions) is considered (see Remark 4.10 in \cite{Haslinger1999}).
\noindent In \cite{Guanghui2010} the authors proved the convergence of the finite difference scheme
(with respect to both time and space variable) for the case of multivalued source term (i.e. $U=H$
in the sequel).

\noindent Our approach uses the so-called Rothe method (known also as time approximation method)
and allows to extend any numerical method that is used to solve the stationary, elliptic inclusions
with the multivalued term given as the Clarke subdifferential, to time dependent, parabolic
problems. The key idea is the replacement of time derivative with the backward difference scheme
and solve the associated elliptic problem in every time step to find the solution in the
consecutive points of the time mesh. It is proved that the results obtained by such approach
approximate the solution of the original problem.

\noindent On the other hand, the Rothe metod provides the proof of existence of solutions. In
contrast to other approaches this metod, as long as one can solve underlying elliptic problems,
does not require any smoothing or other additional regularizing terms in the inclusion. Furthermore
the presented approach allows to study the inclusions with multivalued term given on the domain and
on the domain boundary within the unified framework in which the multifunction that appears in the
problem is defined on an arbitrary reflexive Banach space, which satisfies the appropriate
assumption ($H(U)$ in the sequel). This assumption is proved to generalize the case of inclusions
with multivalued boundary conditions and the ones with multivalued source term (see Section
\ref{sec:lions} and examples of problem settings in Section \ref{sec:settings}).

\noindent The Rothe method for parabolic nonlinear PDEs with pseudomonotone operators is described
in the monograph of Roubicek \cite{Roubicek2005}, where also the results for the monotone
multivalued problems are presented. In the context of parabolic HVIs the variant of the Rothe
method was used to show existence of solutions to problems with hysteresis in \cite{Miettinen1998}
and \cite{Miettinen2003}, but there only the case of linear operator $A$ and $f\in
L^2(0,T;L^2(\Omega))$ (which excludes nonhomogeneous Neumann conditions) was considered and besides
only the case of the multivalued and nonmonotone term source term was analyzed.

\noindent In Section \ref{sec:prel} some basic definitions are recalled. Section \ref{sec:lions}
presents the generalization of the Lions-Aubin Compactness Lemma that justifies the usage of the
assumption $H(U)$ in the sequel. Problem setup and the assumptions are presented in Section
\ref{sec:prblm}. The auxiliary elliptic problems solved in every time step, which are the key idea
of the Rothe method, are formulated and analyzed in Section \ref{sec:rothe}. Convergence of
piecewise linear and piecewise constant functions constructed basing on the solutions of auxiliary
problems as well as the fact that the limit solves the original problem is proved in Section
\ref{sec:limit}. Some stronger convergence and uniqueness results are established in Section
\ref{sec:uniq}. Finally in Section \ref{sec:settings} it is shown that the cases of multivalued
boundary condition and source term are the special cases of presented general framework and a
simple numerical example is delivered.

\section{Preliminaries} \label{sec:prel}
\noindent In this section we recall several key definitions that will be used in the sequel.

\noindent For a locally Lipschitz functional $j:X\to \mathbb{R}$, where $X$ is a Banach space,
generalized directional derivative (in the sense of Clarke) at $x\in X$ in the direction $z\in X$
is defined as
$$j^0(x;z)=\limsup_{y\to x, \lambda \to 0^+}\frac{j(y+\lambda z)-j(y)}{\lambda}.$$ \noindent
Generalized gradient of $j$ (in the sense of Clarke) is the multifunction $\partial j:X\to 2^{X^*}$
defined by
$$
\partial j(x)=\{\xi\in X^*: j^0(x;y)\geq\langle \xi,y\rangle \ \ \mbox{for all}\ \ y\in X\},
$$
\noindent where $\langle\cdot,\cdot\rangle$ stands for the duality pairing between $X$ and $X^*$.
For the properties and the calculus of the Clarke gradient see \cite{Clarke1990}.

\noindent Recall that the multifunction $A:X\to 2^{X^*}$, where $X$ is a real and reflexive Banach
space is pseudomonotone if
\begin{itemize}
\item[(i)] $A$ has values which are nonempty, weakly compact and convex,
\item[(ii)] $A$ is usc from every finite dimensional subsepace of $X$ into $X^*$ furnished with weak topology,
\item[(iii)] if $v_n\to v$ weakly in $X$ and $v_n^*\in A(v_n)$ is such that $\limsup_{n\to\infty}\langle v_n^*,v_n-v\rangle \leq 0$ then for every $y\in X$ there exists $u(y)\in A(v)$ such that $\langle u(y),v-y\rangle\leq \liminf_{n\to\infty}\langle v_n^*,v_n-y\rangle$.
\end{itemize}
\noindent Note that sometimes it is useful to check the pseudomonotonicity of an operator via the
following sufficient condition (see Proposition 1.3.66 in \cite{Denkowski2003} or Proposition 3.1
in \cite{Carl2006}).
\begin{prop} \label{prop:pseudo} Let X be a real reflexive Banach space, and assume that $A: X\to 2^{X^*}$ satisfies the following conditions
\begin{itemize}
\item[$(i)$] for each $v\in X$ we have that $A(v)$ is a nonempty, closed and convex subset of $X^*$.
\item[$(ii)$] $A$ is bounded.
\item[$(iii)$] If $v_n\to v$ weakly in $X$ and $v_n^*\to v^*$ weakly in $X^*$  with $v_n^*\in A(v_n)$ and if $\limsup_{n\to\infty}\langle v_n^*, v_n - v\rangle \leq 0$, then $v^*\in A(v)$ and $\langle v_n^*,v_n\rangle \to \langle v^*,v\rangle$.
\end{itemize}
Then the operator $A$ is pseudomonotone.
\end{prop}
\noindent We also recall (see for instance Proposition 1.3.68 \cite{Denkowski2003}) that the sum of
two pseudomonotone multifunctions is pseudomonotone.

\section{Generalization of Lions-Aubin Lemma}\label{sec:lions}
\noindent For a Banach space $X$, $1\leq p\leq \infty$ and a finite time interval $I=(0,T)$ we
consider the standard spaces $L^p(I;X)$. Furthermore we denote by $BV(I;X)$ the space of functions
of bounded total variation on $I$. Let $\pi$ denote any finite partition of $I$ by a family of
disjoint subintervals $\{\sigma_i=(a_i,b_i)\}$  such that $\bar{I}=\bigcup_{i=1}^n\bar{\sigma}_i$.
Let $\C{F}$ denote the family of all such partitions. Then we define the total variation as
$$
\|x\|_{BV(I;X)} = \sup_{\pi\in \C{F}}\left\{\sum_{\sigma_i\in\pi}\|x(b_i)-x(a_i)\|_X\right\}.
$$
As a generalization of above definition for $1\leq q < \infty$ we can define a seminorm
$$
\|x\|_{BV^q(I;X)}^q = \sup_{\pi\in \C{F}}\left\{\sum_{\sigma_i\in\pi}\|x(b_i)-x(a_i)\|^q_X\right\}.
$$
For Banach spaces $X,Z$ such that $X\subset Z$ we introduce a vector space
\begin{equation}
M^{p,q}(I;X,Z)= L^p(I;X)\cap BV^q(I;Z).\nonumber
\end{equation}
Then $M^{p,q}(I;X,Z)$ is also a Banach space for
$1\leq p,q <\infty$ with the norm given by $\|\cdot\|_{L^p(I;X)}+\|\cdot\|_{BV^q(I;Z)}$.




\noindent Let us recall Theorem 1 of \cite{Simon1987} (see also Theorem 1 of \cite{Rossi2003} and
Proposition 2.1 of \cite{Rossi2001}).
\begin{thm}\label{thm:simon}
Let $1\leq p<\infty$ and $X$ be a real Banach space. A subset $\C{G}\subset L^p(0,T;X)$ is
relatively compact in a Banach space $L^p(0,T;X)$ provided the following two conditions hold
\begin{itemize}
\item for every $0<t_1<t_2<T$ the set
$$
G(t_1,t_2) := \left\{\int_{t_1}^{t_2}u(t)\, dt: u\in \C{G} \right\}
$$
is relatively compact in $X$,
\item $\C{G}$ is strongly integrally equicontinuous i.e.
\begin{equation}
\lim_{h\to 0}\sup_{u\in \C{G}}\int_0^{T-h}\|u(t+h)-u(t)\|_{X}^p\ dt = 0.
\end{equation}
\end{itemize}
\end{thm}

\noindent The following proposition is a consequence of Theorem \ref{thm:simon}.
\begin{prop}\label{prop:compactembedding}
Let $1\leq p,q<\infty$. Let $X_1\subset X_2\subset X_3$ be real Banach spaces such that $X_1$ is
reflexive, the embedding $X_1\subset X_2$ is compact and the embedding $X_2\subset X_3$ is
continuous. If a subset $\C{G}\subset M^{p,q}(I; X_1,X_3)$ is bounded, then it is relatively
compact in $L^p(I;X_2)$.
\end{prop}
\begin{proof}
\noindent We apply Theorem \ref{thm:simon} with $X=X_2$. Let us fix $0<t_1<t_2<T$ and let $v\in
G(t_1,t_2)$. For $u \in \C{G}$ we have
\begin{eqnarray}
\|v\|_{X_1}=\left\|\int_{t_1}^{t_2}u(t)\, dt\right\|_{X_1}\leq \int_{0}^{T} \|u(t)\|_{X_1}\, dt\leq
T^{1-\frac{1}{p}}\|u\|_{L^p(I;X_1)}.
\end{eqnarray}
Thus $G(t_1,t_2)$ is bounded in $X_1$ and therefore relatively compact in $X_2$.

\noindent It suffices to show the strong integral equicontinuity of $\C{G}$. Let $\sup_{u\in
\C{G}}\|u\|_{L^p(I;X_1)}^p=M$. We will use the Ehrling Lemma (see for instance \cite{Roubicek2005},
Lemma 7.6). Let us fix $\varepsilon>0$.  There exists $C>0$ such that for $v\in X_1$ we have
$\|v\|_{X_2}^p\leq \frac{\varepsilon}{2^pM}\|v\|_{X_1}^p+C\|v\|_{X_3}^p$. In particular, fixing
$h\in (0,T)$, for $u\in \C{G}$ and almost every $t\in (0,T-h)$ we have $\|u(t+h)-u(t)\|_{X_2}^p\leq
\frac{\varepsilon}{2^pM}\|u(t+h)-u(t)\|_{X_1}^p+C\|u(t+h)-u(t)\|_{X_3}^p$. Integrating this
inequality we get
\begin{eqnarray}
&&\int_0^{T-h}\|u(t+h)-u(t)\|_{X_2}^p\ dt\leq \nonumber\\
&&\leq\frac{\varepsilon}{2^pM}\int_0^{T-h}\|u(t+h)-u(t)\|_{X_1}^p\ dt+C\int_0^{T-h}\|u(t+h)-u(t)\|_{X_3}^p\ dt \leq \label{ineq:equicont}\nonumber\\
&&\leq \frac{\varepsilon}{2M}\int_0^{T-h}\|u(t+h)\|_{X_1}^p+\|u(t)\|_{X_1}^p\ dt+C\int_0^{T-h}\|u(t+h)-u(t)\|_{X_3}^p\ dt\leq \nonumber\\
&&\leq \varepsilon + C\int_0^{T-h}\|u(t+h)-u(t)\|_{X_3}^p\ dt.
\end{eqnarray}
Now let $\sup_{u\in \C{G}}\|u\|_{BV^q(I;X_3)}^q=S$. If $p \leq q$, then by the H\"{o}lder
inequality, we have
\begin{equation}\label{ineq:case1}\int_0^{T-h}\|u(t+h)-u(t)\|_{X_3}^p\ dt \leq T^{1-\frac{p}{q}}
\left(\int_0^{T-h}\|u(t+h)-u(t)\|_{X_3}^q\ dt\right)^{\frac{p}{q}}.\end{equation} If in turn $q<p$,
then
\begin{equation}
\label{ineq:case2}\int_0^{T-h}\|u(t+h)-u(t)\|_{X_3}^p\ dt
\leq S^{\frac{p}{q}-1}\int_0^{T-h}\|u(t+h)-u(t)\|_{X_3}^q\ dt.
\end{equation}
We estimate the last term in (\ref{ineq:case1}) and (\ref{ineq:case2}) from above (taking, if
necessary, $u(t)=u(T)$, if $t>T$)
\begin{eqnarray}
&& \int_0^{T-h}\|u(t+h)-u(t)\|_{X_3}^q\ dt\leq \sum_{i=0}^{\lceil T/h-2\rceil}\int_{ih}^{ih+h}\|u(t+h)-u(t)\|_{X_3}^q\ dt=\nonumber\\
&&=\sum_{i=0}^{\lceil T/h-2\rceil}\int_{0}^{h}\|u(t+ih+h)-u(t+ih)\|_{X_3}^q\ dt=\nonumber\\
&&=\int_{0}^{h}\sum_{i=0}^{\lceil T/h-2\rceil}\|u(t+ih+h)-u(t+ih)\|_{X_3}^q\ dt\leq Sh.
\end{eqnarray}
Thus the last term in (\ref{ineq:equicont}) tends to $0$ uniformly in $u$ as $h\to 0$ and, since
$\varepsilon$ was arbitrary, we get the thesis.
\end{proof}
\noindent\textbf{Remark 1.} Note, that Theorem 3.2 in \cite{Ahmed2003} is a consequence of above
theorem. Compare also the Corollary 7.9 in \cite{Roubicek2005} where the case $p=1$ is excluded and
$X_3$ is assumed to have a predual space.
\section{Problem formulation and assumptions}\label{sec:prblm}
\noindent Let $V\subset H\subset V^*$ be an evolution triple, where $V$ is a reflexive and separable Banach space
and $H$ is a separable Hilbert space with the embeddings being continuous, dense and compact. Embedding
between $V$ and $H$ will be denoted by $i$. Furthermore let $U$ be a reflexive Banach space on
which the multivalued term will be defined. We use the notation $\C{V}=L^2(0,T;V)$,
$\C{H}=L^2(0,T;H)$, $\C{U}=L^2(0,T;U)$ and $\C{W}=\{u\in \C{V}, u' \in \C{V}^*\}$, where the
derivative is understood in the sense of distibutions. Duality parings and norms for all the spaces
will be denoted by the appropriate subscripts, for the space $V$ no subscript will be used. Scalar
product in $H$ will be denoted by $(\cdot,\cdot)$ and norm in $\mathbb{R}^n$ by $|\cdot|$. We
consider the operator $A:V\to V^*$ and the functional $J:U\to \mathbb{R}$ such that the following
assumptions hold
\begin{itemize}
\item[$H(A)$:]
\begin{itemize}
\item[$(i)$] $A$ is pseudomonotone,
\item[$(ii)$] $A$ satisfies the growth condition $\|A(v)\|_{V^*}\leq a+b\|v\|$ for every $v\in V$ with $a\geq 0, b>0$,
\item[$(iii)$] $A$ is coercive $\langle A(v),v\rangle\geq \alpha \|v\|^2-\beta\|v\|^2_{H}$ for every $v\in V$ with $\alpha>
0$ and $\beta\geq 0$,
\end{itemize}
\item[$H(J)$:]\begin{itemize}
\item[$(i)$] $J$ is locally Lipschitz,
\item[$(ii)$] $\partial J$ satisfies the growth condition $\|\xi \|_{U^*}\leq c(1+\|u\|_U)$ for every $u\in U$ and $\xi \in \partial J(u)$ with $c>0$.
\end{itemize}
\end{itemize}
\noindent Moreover we assume that
\begin{itemize}
\item[$H_0$:]$f\in \C{V}^*$ and $u_0\in H$.
\end{itemize}
\noindent We also impose the assumption concerning the space $U$
\begin{itemize}
\item[$H(U)$:] There exists the linear, continuous and compact mapping $\iota:V\to U$ such that the associated Nemytskii mapping $\bar{\iota}:M^{2,2}(0,T;V,V^*)\to \C{U}$ defined by $(\bar{\iota} v)(t)=\iota(v(t))$ is also compact.
\end{itemize}
\noindent Finally we impose the last assumption
\begin{itemize}
\item[$H_{aux}$:]One of the following holds
\begin{itemize}
\item[A)] There exists a linear and continuous mapping $p:H\to U$ such that for $v\in V$ we have $p(i(v))=\iota(v)$.
\item[B)] The constants $\alpha$ and $c$ satisfy the inequality $\alpha > c \|\iota\|_{\C{L}(V;U)}^2$.
\item[C)] For every $u\in U$ we have $J^0(u;-u)\leq d(1+\|u\|_U^\sigma)$ with $d\geq 0$ and $1\leq \sigma <2$.
\end{itemize}
\end{itemize}
\noindent The problem under consideration is as follows
\begin{eqnarray}
&& \mbox{find}\  u\in \C{W}\ \mbox{such that}\  u(0)=u_0 \ \mbox{and for a.e.}
\ t\in (0,T)\ \mbox{we have}\nonumber\\
&&u'(t)+Au(t)+\iota^* \partial J (\iota u(t)) \ni f(t).\label{eq:inclusionmain}
\end{eqnarray}
The last inclusion is understood in the following sense
\begin{eqnarray}
&&\mbox{there exists}\ \  \eta\in \C{V}^* \ \mbox{such that}\ u'(t)+Au(t)+\eta(t) = f(t)\ \mbox{for a.e.}\ t\in(0,T)\nonumber\\
&&\mbox{and}\ \ \langle\eta(t),v\rangle\in \langle \partial J(\iota u(t)),\iota v\rangle_{U^*\times U} \
\mbox{for a.e.}\ t\in(0,T)\ \mbox{and}\ v\in V.
\end{eqnarray}
\noindent\textbf{Remark 2.} The formulation (\ref{eq:inclusionmain}) puts into a unified framework
hemivariational inequalities originating from the initial and boundary value problem with
multivalued term defined on the problem domain (in this case we have multivalued source term, and
$U=H$, see \cite{Miettinen1996,Miettinen1999,Migorski2001}) and on the part $\Gamma_C$ of domain
boundary $\partial\Omega$ (this is the case if we have the multivalued, nonlinear and nonmonotone
boundary condition of Neumann-Robin type, $U=L^2(\Gamma_C)$ or $U=L^2(\Gamma_C;\mathbb{R}^n)$, see
\cite{Migorski2004}). A detailed discussion as well as examples of problems which satisfy the
assumptions will be given in Section \ref{sec:settings}.

\noindent\textbf{Remark 3.} For the sake of simplicity of further argument the assumptions given
above are not the most general ones under which the results hold. Possible generalizations include:
\begin{itemize}
\item The dependance of $A$ and $J$ on time variable. Time dependent operator $A$ for parabolic HVI is considered in \cite{Miettinen1999} and the case of both $A$ and $J$ depending on time is considered in \cite{Migorski2004} (see Remark 8.21 in \cite{Roubicek2005} on the Rothe method for the problem with the operator depending on time).
\item Instead of pseudomonotonicity one could assume that $A$ is a sum of two operators, one of which is pseudomonotone and the second one is weakly continuous. Such weak continuity allows to take into account the nonlinear terms of lower order which are not of monotone type (see \cite{Francu1994}).
\item More general coercivity conditions on $A$ can be assumed. For instance $\langle Av,v\rangle\geq c\|v\|^2-a\|v\|_V-\gamma(t)$ with $c>0$, $a\geq 0$ and $\gamma\in L^1(0,T)$ cf. \cite{Migorski2001}.
\item The case when the space $\C{V}$ is defined as $L^p(0,T;V)$ with $2<p<\infty$ can be considered. Then we can assume more general growth conditions on $A$ and $J$. For instance in \cite{Migorski2001} it is assumed that $\|A(t,v)\|_{V^*}\leq \beta(t)+c_1\|v\|^{p-1}$ and that for $\eta\in \partial j(x,\xi)$ we have $|\eta|\leq c(1+|\xi|^{p-1})$.
Note that $J$ is defined typically as the integral functional $J(u)=\int_{\omega}j(x,u(x))\, dx$
and assumptions on the integrand $j$ are given.
\end{itemize}
\noindent \textbf{Remark 4.} In this paper the abstract setting is considered. For a divergence
differential operator of Leray - Lions type on a Sobolev space pseudomonotnicity is implied by the
appropriate Leray - Lions type conditions (see, for instance, \cite{Berkovits1996} where conditions
that guarantee pseudomonotonicity on $W^{m,p}(\Omega)$, $1<p<\infty$, $m\geq 1$ are considered).

\noindent We conclude this section with the Lemma on pseudomonotonicity of Nemytskii operator with respect to the space $M^{2,2}(0,T;V,V^*)$. Note that the proof of this lemma is analogous to the proof of Theorem 2 (b) in \cite{Berkovits1996} (see also Proposition 1 from \cite{Papageorgiou1997} and Lemma 8.8 in \cite{Roubicek2005} for similar results). Lemma 8.8 of \cite{Roubicek2005} is most similar to Lemma \ref{lem:nemytskii}, but note that here no a priori bound in $L^\infty(0,T;H)$ is needed and the assumption on the bound of 2-variation which is used here is weaker then the bound on 1-variation as in \cite{Roubicek2005}.
\begin{lemma}\label{lem:nemytskii}
Let $A: V\to V^*$ satisfy $H(A)$ and let $\C{A}:\C{V}\to \C{V}^*$ be a Nemytskii operator for $A$ defined by $(\C{A}u)(t)=A(u(t))$. Then if, for a uniformly bounded sequence $\{u_n\}\subset M^{2,2}(0,T;V,V^*)$ such that $u_n\to u$ weakly in $\C{V}$ we have $\limsup_{n\to\infty} \langle \C{A}u_n,u_n-u\rangle_{\C{V}^*\times \C{V}}\leq 0$, then $\C{A}u_n \to \C{A}u$ weakly in $\C{V}^*$.
\end{lemma}
\begin{proof} It is enough to show that the thesis holds for a subsequence. By the generalized Lions Aubin Compactness Lemma (see Proposition \ref{prop:compactembedding}) for a subsequence (still denoted by $n$) we have $u_n\to u$ strongly in $\C{H}$. Moreover, for yet another subsequence $u_n(t)\to u(t)$ strongly in $H$ for a.e. $t\in (0,T)$. We denote the set of measure zero on which the convergence does not hold by $N$. Now let us define $\xi_n(t)=\langle Au_n(t),u_n(t)-u(t)\rangle$. We have
\begin{eqnarray}\label{eq:boundbelowxi}
&& \xi_n(t)\geq \alpha\|u_n(t)\|^2-\beta\|u_n(t)\|_H^2-\|u(t)\|(a+b\|u_n(t)\|)\geq\\
&&\geq \frac{\alpha}{2}\|u_n(t)\|^2-\beta \|u_n(t)\|_H^2 - a\|u(t)\| - \frac{b^2}{2\alpha}\|u(t)\|^2.\nonumber
\end{eqnarray}
Now let $C=\{t \in [0,T]: \liminf_{n\to\infty}\xi_n(t)< 0\}$. This is the Lebesgue measurable subset of $[0,T]$. Suppose that $m(C)>0$, $m$ being one dimensional Lebesgue measure. For every $t\in C\setminus N$ the sequence $u_n(t)$ has a subsequence (still denoted by $n$) which is bounded in $V$ by (\ref{eq:boundbelowxi}) such that $\lim_{n\to\infty}\langle Au_n(t),u_n(t)-u(t) \rangle<0$. Again for a subsequence we have  $u_n(t)\to u(t)$ weakly in $V$, where the limit equals $u(t)$ since we can consider only $t\notin N$. By the pseudomonotonicity of $A$ we get $0\leq \liminf_{n\to \infty} \langle Au_n(t), u_n(t)-u(t)\rangle$, which is a contradiction. So $m(C)=0$, which means that $\liminf_{n\to\infty} \xi_n(t)\geq 0$ a.e. on $(0,T)$. From the Fatou Lemma we have 
\begin{eqnarray}
&&\beta\|u\|_H^2 \leq \int_0^T\liminf_{n\to\infty} \xi_n(t)\, dt + \beta\|u\|_H^2 \leq  \int_0^T \liminf_{n\to\infty} (\xi_n(t)+\beta\|u_n(t)\|_H^2) \, dt \leq \nonumber \\
&& \leq \liminf_{n\to\infty } \int _0^T  \xi_n(t)+ \beta\|u_n(t)\|_H^2 \, dt \leq \nonumber\\
&& \leq \liminf_{n\to\infty } \int _0^T  \xi_n(t)\, dt + \beta \|u\|_H^2\leq \limsup_{n\to\infty } \int _0^T  \xi_n(t) \, dt + \beta \|u\|_H^2 \leq \beta \|u\|_H^2.\nonumber
\end{eqnarray}
So $\int_0^T\xi_n(t)\, dt\to 0$ as $n\to \infty$. Now note that $|\xi_n(t)|=\xi_n(t)+2\xi_n^-(t)$ and $\xi_n^-(t)\to 0$ for a.e. $t\in (0,T)$. Since, by (\ref{eq:boundbelowxi}), for a.e. $t\in (0,T)$ we have $\xi_n(t) + \beta \|u_n(t)\|_H^2 \geq f(t)$ with $f \in L^1(0,T)$, then $ \xi_n^-(t)-\beta\|u_n(t)\|_H^2 \leq f^-(t)$. Invoking Fatou Lemma again we have $\limsup\int_0^T\xi^-(t)\, dt \leq 0$ and furthermore $\int_0^T\xi_n^-(t)\, dt \to 0$ as $n\to \infty$. We deduce that $\xi_n\to 0$ in $L^1(0,T)$ and, for a subsequence (still denoted by the same subscript), $\xi_n(t)\to 0$ for a.e. $t\in (0,T)$. Since, for this subsequence, $u_n(t)\to u(t)$ weakly in $V$, then by pseudomonotonicity of $A$ it follows that $Au_n(t)\to Au(t)$ weakly in $V^*$ and $\langle Au_n(t), u_n(t)\rangle \to \langle Au(t), u(t) \rangle$. For any $v\in \C{V}$ we have 
\begin{eqnarray}
&& \langle \C{A}u, u-v\rangle_{\C{V}^*\times \C{V}} = \int_0^T\langle Au(t), u(t)-v(t) \rangle\, dt = \int_0^T\lim_{n\to\infty} \langle Au_n(t), u_n(t)-v(t)\rangle\, dt=\nonumber \\
&&= -\beta\|u\|_H^2 + \int_0^T\lim_{n\to\infty} (\langle Au_n(t), u_n(t)-v(t)\rangle + \beta\|u_n(t)\|_H^2)\, dt .\nonumber
\end{eqnarray}
We can apply Fatou Lemma one last time to get 
\begin{eqnarray}
&& \langle \C{A}u, u-v\rangle_{\C{V}^*\times \C{V}} \leq \liminf_{n\to\infty}\int_0^T\langle Au_n(t), u_n(t)-v(t)\rangle\, dt =\nonumber\\
&&= \liminf_{n\to\infty} (\langle \C{A}u_n, u_n-u\rangle_{\C{V}^*\times \C{V}}+\langle \C{A}u_n, u-v\rangle_{\C{V}^*\times \C{V}})\leq\nonumber \\
&&\leq\liminf_{n\to\infty} \langle \C{A}u_n, u-v\rangle_{\C{V}^*\times \C{V}}.
\end{eqnarray}
Since $v$ is arbitrary we obtain the thesis.
\end{proof}

\section{The Rothe problem}\label{sec:rothe}
\noindent In this section we will work with a sequence of time-steps $\tau_n\to 0$ such that each
time step $\tau_n>0$ and the value $T/\tau_n$ is an integer, which we denote by $N_n$. The subscipt
$n$ will be omitted in the sequel in order to simplify the notation, so we will write $N, \tau$
instead of $N_n, \tau_n$.

\noindent We define the piecewise constant approximation of the function $f\in \C{V}^*$. For this
purpose we take the sequence of positive numbers $\epsilon(\tau)\to 0$ and the sequence of
mollifiers $\rho_\epsilon:\mathbb{R}\to\mathbb{R}$ which belong to $C^\infty(\mathbb{R})$ and are nonnegative, supported on
$[-\epsilon,\epsilon]$ and $\int_{\mathbb{R}}\rho_{\epsilon}(x)\, dx = 1$. The function $f$ is
regularized according to the formula
$$f_\epsilon(t)=\int_0^T \rho_\epsilon\left(t+\epsilon\frac{T-2t}{T}-s\right)f(s)\, ds.$$
Note that $f_\epsilon\in C^1(0,T;V^*)$ (see \cite{Roubicek2005}, Lemma 7.2). The piecewise constant
approximation for $f$ is given by
$$
\bar{f}_\tau(t):=f^k_\tau=f_{\epsilon(\tau)}(k\tau)\ \mbox{for}\ t\in((k-1)\tau,k\tau],\
k\in\{1,\ldots,N\}.
$$
Following \cite{Roubicek2005}, Lemma 8.7, we have $\bar{f}_\tau\to f$ in $\C{V}^*$ when $\tau\to
0$. Note (see Remark 8.15 in \cite{Roubicek2005}) that the smoothing of $f$ is not the only
possible approach here. It is also possible to take the Cl\'{e}ment zero-order quasi interpolant
$f_\tau^k=\frac{1}{\tau}\int_{(k-1)\tau}^{k\tau}f(\theta)\,d\theta$.

\noindent We approximate the initial condition by elements of $V$. Let $\{u_{0\tau}\}\subset V$ be
a sequence such that $u_{0\tau}\to u_0$ strongly in $H$ and $\|u_{0\tau}\|\leq C/\sqrt{\tau}$ for
some constant $C>0$.

\noindent We define the following Rothe problem
\begin{eqnarray}
&& \mbox{find the sequence}\ \{u_\tau^k\}_{k=0}^N \subset V\  \mbox{such that}\  u_\tau^0=u_{0\tau} \ \mbox{and}\nonumber\\
&& \left(\frac{u^{k}_\tau-u^{k-1}_\tau}{\tau},v\right)_H+\langle Au_\tau^k,v\rangle+
\langle\partial J (\iota u^k_\tau), \iota v\rangle_{U^*\times U} \ni \langle f^k_\tau, v\rangle
\label{eq:rotheproblem}\\
&& \mbox{for all}\  v\in V\  \mbox{and}\ k=1,\ldots,N.\nonumber
\end{eqnarray}
The above formula is known as the implicit or backward Euler scheme. Existence of solutions to the
Rothe problem follows from the following
\begin{lemma}
Under assumptions $H(A), H(J), H_0, H(U)$ and $H_{aux}$ there exists $\tau_0>0$ such that the
problem (\ref{eq:rotheproblem}) has a solution for $\tau\in (0,\tau_0)$.
\end{lemma}
\begin{proof}
We show that, given $u^{k-1}_\tau \in V$, we can find $u^k_\tau\in V$ such that
(\ref{eq:rotheproblem}) holds. We need to show that the range of multifunction $V\ni v\to
Lv=\frac{i^* iv}{\tau}+Av+\iota^*
\partial J(\iota v)\in 2^{V^*}$ constitutes the whole space $V^*$. We will use the surjectivity
theorem for pseudomonotone operators (see for instance Theorem 1.3.70 in \cite{Denkowski2003}). We
need to show that $L$ is coercive (in the sense that $\lim_{\|v\|\to\infty}\frac{\inf_{v^*\in
Lv}\langle v^*,v\rangle}{\|v\|}=\infty$) and pseudomonotone.

\noindent \textbf{Claim 1.} \textit{$L$ is pseudomonotone.} We verify this condition for all
components of $L$ separately. For this purpose we use Proposition \ref{prop:pseudo}. The operator
$\frac{i^*i}{\tau}$ satisfies the conditions $(i)-(iii)$ trivially. As for $\iota^* \partial
J(\iota u)$ the condition $(i)$ follows from the fact that the Clarke subdifferential has nonempty,
convex and (for reflexive space) weakly compact values. The condition $(ii)$ follows from the
growth assumption on $\partial J$. In order to verify $(iii)$ let us take $v_n\to v$ weakly in $V$
and $\xi_n\to \xi$ weakly in $V^*$ with $\xi_n\in \iota^*\partial J(\iota v_n)$. Obviously $\iota
v_n\to \iota v$ strongly in $U$. Define $\eta_n \in \partial J(\iota v_n)$ such that $\xi_n=\iota^*\eta_n$. By the growth condition $H(J)(ii)$ it follows that, for a subsequence still denoted by the same subscript, $\eta_n\to \eta$ weakly in $U^*$. By the closedness of the
graph of $\partial J$ in $U\times U^*_{w}$ topology (see \cite{Clarke1990}, Proposition 2.1.5), we
get $\eta \in \partial J(\iota v)$. Obviously $\xi=\iota^*\eta$ and $\xi\in \iota^*\partial J(\iota v)$. Moreover $\langle v_n,\xi_n\rangle=\langle \iota v_n,
\eta_n\rangle_{U^*\times U} \to \langle \iota v, \eta\rangle_{U^*\times U}=\langle v,\xi\rangle$, where by uniqueness convergence holds for the whole sequence.

\noindent \textbf{Claim 2.} \textit{$L$ is coercive.} Assume that $v^*\in Lv$. We estimate $\langle
v^*,v\rangle$ from below. For some $\eta\in \partial J(\iota v)$ we have
\begin{equation}\label{eq:estimatecorec}
\langle v^*,v \rangle \geq \frac{1}{\tau}\|v\|_H^2+\alpha \|v\|^2 - \beta \|v\|_H^2+\langle \eta,
\iota v\rangle_{U^*\times U}.
\end{equation}
We proceed for cases $A), B), C)$ separately. For $A)$ and $B)$, by the growth condition
$$
\langle v^*,v \rangle \geq \left(\frac{1}{\tau}-\beta\right)\|v\|_H^2+\alpha \|v\|^2-c(1+\|\iota
v\|_U)\|\iota v\|_U.
$$
In the case $A)$ we have $\|\iota v\|_U^2\leq \|p\|_{\C{L}(H,U)}^2\|v\|_H^2$, so
$$
\langle v^*,v \rangle \geq \left(\frac{1}{\tau}-\beta-c\|p\|_{\C{L}(H,U)}^2\right)\|v\|_H^2+\alpha
\|v\|^2-c\|\iota\|_{\C{L}(V;U)}\|v\|.
$$
We require $\tau_0 = \frac{1}{\beta+c\|p\|_{\C{L}(H,U)}}$. In the case $B)$ we get
$$
\langle v^*,v \rangle \geq
\left(\frac{1}{\tau}-\beta\right)\|v\|_H^2+(\alpha-c\|\iota\|^2_{\C{L}(V;U)})\|v\|^2-c\|\iota\|_{\C{L}(V;U)}\|v\|.
$$
To have coercivity we need to set $\tau_0=\frac{1}{\beta}$. Finally if $C)$ holds, then we get
\begin{eqnarray}
&&\langle \eta, \iota v\rangle_{U^*\times U}\geq -J^0(\iota v;-\iota v) \geq\nonumber\\
&&\geq - d(1+\|\iota
v\|_U^\sigma)\geq -d -\|\iota\|_{\C{L}(V;U)}^\sigma\|v\|^\sigma\geq -d
-\frac{\alpha}{2}\|v\|^2-C,\nonumber\end{eqnarray}
where $C>0$ depends on $\alpha, \sigma$ and $\|\iota\|_{\C{L}(V;U)}$.
Combining the last estimate with (\ref{eq:estimatecorec}) we get
$$
\langle v^*,v \rangle \geq \left(\frac{1}{\tau}-\beta\right)\|v\|_H^2+\frac{\alpha}{2} \|v\|^2 - d
- C.
$$
Again setting $\tau_0=\frac{1}{\beta}$ we get the desired property.
\end{proof}

\noindent Next lemma establishes the estimates which are satisfied by the solutions of Rothe
problem.
\begin{lemma}\label{lem:bounds}
Under assumptions $H(A), H(J), H_0, H(U)$ and $H_{aux}$ there exists $\tau_0>0$ such that for all
$\tau\in (0,\tau_0)$ the solutions of Rothe problem (\ref{eq:rotheproblem}) satisfy
\begin{eqnarray}
&&\max_{k=1,\ldots,N}\|u_\tau^k\|_H\leq\mbox{const},\label{eq:bound1}\\
&&\sum_{k=1}^N\|u_\tau^k-u_\tau^{k-1}\|_H^2\leq\mbox{const},\\
&&\tau \sum_{k=1}^N\|u_\tau^k\|^2\leq\mbox{const},\label{eq:bound3}
\end{eqnarray}
with the constants independent on $\tau$.
\end{lemma}
\begin{proof} We take $v=u^k_\tau$ in (\ref{eq:rotheproblem}), which gives for $\varepsilon > 0$ and $k=1,\ldots,N$
\begin{eqnarray}
&& \frac{1}{\tau}\|u^k_\tau\|^2_H+\alpha\|u^k_\tau\|^2+\langle \xi_\tau^k,\iota
u^k_\tau\rangle_{U^*\times U}\leq\nonumber\\
&&\leq \beta\|u^k_\tau\|_H^2+\frac{1}{2\varepsilon}\|f^k_\tau\|^2_{V^*}+\frac{\varepsilon}{2}\|u^k_\tau\|^2+\frac{1}{\tau}(u^{k-1}_\tau,u^k_\tau)\nonumber
\end{eqnarray}
with $\xi_\tau^k\in \partial J(\iota u^k_\tau)$. We use the relation $\|a\|^2-(a,b)=\|a\|^2/2 -
\|b\|^2/2+\|a-b\|^2/2$ to obtain
\begin{eqnarray}\label{eq:estimateonestep}
&&\left(\frac{1}{2\tau}-\beta\right)\|u^k_\tau\|^2_H+\frac{1}{2\tau}\|u^{k}_\tau-u^{k-1}_\tau\|_H^2+\left(\alpha-\frac{\varepsilon}{2}\right)\|u^k_\tau\|^2+\\
&&+\langle
\xi_\tau^k,\iota u^k_\tau\rangle_{U^*\times U}\leq\frac{1}{2\varepsilon}\|f^k_\tau\|^2_{V^*}+\frac{1}{2\tau}\|u^{k-1}_\tau\|_H^2.\nonumber
\end{eqnarray}
Recall that
$$\langle \xi^k_\tau, \iota u^k_\tau\rangle_{U^*\times U}\geq \begin{cases}-C_1-c\|p\|^2_{\C{L}(H,U)}\|u^k_\tau\|^2_H-\frac{\alpha}{2}\|u^k_\tau\|^2& \mbox{if $A)$ holds,}\\ (-c\|\iota\|^2_{\C{L}(U,V)}-\delta)\|u^k_\tau\|^2-C_2& \mbox{if $B)$ holds,}\\-C_3-\frac{\alpha}{2}\|u^k_\tau\|^2 & \mbox{if $C)$ holds,}\end{cases}$$
where $C_1>0$ depends on $c,\alpha, \|\iota\|_{\C{L}(V;U)}$, $\delta>0$ is arbitrary, $C_2>0$
depends on $c, \delta, \|\iota\|_{\C{L}(V;U)}$ and $C_3>0$ depends on $d, \alpha, \sigma,
\|\iota\|_{\C{L}(V;U)}$. From now on we proceed separately for the cases $A), B)$ and $C)$. In the
case $A)$ we take $\varepsilon = \frac{\alpha}{2}$ to get
\begin{eqnarray}
&&\left(\frac{1}{2\tau}-\beta-c\|p\|^2_{\C{L}(H,U)}\right)\|u^k_\tau\|^2_H+\frac{1}{2\tau}\|u^{k}_\tau-u^{k-1}_\tau\|_H^2+\\
&&+\frac{\alpha}{4}\|u^k_\tau\|^2\leq
\frac{1}{\alpha}\|f^k_\tau\|^2_{V^*}+\frac{1}{2\tau}\|u^{k-1}_\tau\|_H^2+C_1.\nonumber
\end{eqnarray}
Summing above inequalities for $k=1,\ldots,n$, where $1\leq n\leq N$, we have
\begin{eqnarray}
&&\|u^n_\tau\|_H^2+\sum_{k=1}^n\|u^k_\tau-u^{k-1}_\tau\|_H^2+\frac{\alpha\tau}{2}\sum_{k=1}^n\|u_\tau^k\|^2\leq\\
&&\leq 2TC_1 + 2 \tau (\beta + c\|p\|^2_{\C{L}(H,U)})
\sum_{k=1}^n\|u^k_\tau\|_H^2+\frac{2\|\bar{f}_\tau\|^2_{\C{V}^*}}{\alpha}+\|u^0_\tau\|_H^2.\nonumber
\end{eqnarray}
Now if $\tau < 1 / (4(\beta + c\|p\|^2_{\C{L}(H,U)}))$, by a discrete Gronwall inequality (see e.g.
\cite{Roubicek2005} (1.68)-(1.69)), we have (\ref{eq:bound1})-(\ref{eq:bound3}).

\noindent In the case $B)$, for $\delta = \varepsilon =
(\alpha-c\|\iota\|^2_{\C{L}(U,V)})/2$ we get
\begin{eqnarray}
&&\left(\frac{1}{2\tau}-\beta\right)\|u^k_\tau\|^2_H+\frac{1}{2\tau}\|u^{k}_\tau-u^{k-1}_\tau\|_H^2+\frac{\alpha-c\|\iota\|^2_{\C{L}(U,V)}}{4}\|u^k_\tau\|^2\leq\\
&&\leq C_4\|f^k_\tau\|^2_{V^*}+\frac{1}{2\tau}\|u^{k-1}_\tau\|_H^2+C_5,\nonumber
\end{eqnarray}
where $C_4 = \frac{1}{\alpha-c\|\iota\|^2_{\C{L}(U,V)}}$ and $C_5>0$ depends on $\alpha, c,
\|\iota\|^2_{\C{L}(U,V)}$. In analogy to the previous case we get
(\ref{eq:bound1})-(\ref{eq:bound3}) for $\tau < 1/\beta$. Bounds in the case $C)$ are obtained in
an analogous way.
\end{proof}

\section{Convergence of the Rothe method}\label{sec:limit}
\noindent We define piecewise linear and piecewise constant interpolants $u_\tau \in C([0,T];V)$
and $\bar{u}_\tau\in L^\infty(0,T; V)$ by the formulae
\begin{eqnarray}
&&u_\tau(t) = \left(\frac{t}{\tau}-k+1\right)u^k_\tau + \left(k-\frac{t}{\tau}\right)u^{k-1}_\tau\ \mbox{for}\  t\in[(k-1)\tau,k\tau],\nonumber\\
&&\bar{u}_\tau(t) = u^k_\tau\ \mbox{for a.e.}\  t\in((k-1)\tau,k\tau].\nonumber
\end{eqnarray}
where $k=1,\ldots,T/\tau$.

The sequences $\{u_{\tau_n}\}_{n=1}^\infty$ and $\{\bar{u}_{\tau_n}\}_{n=1}^\infty$ are known as
the Rothe sequences. Observe, that $u_\tau$ has a distributional derivative $u'_\tau\in
L^\infty(0,T;V)$ given by $u'_\tau(t) = \frac{u^k_\tau-u^{k-1}_\tau}{\tau}$ for almost every $t\in
((k-1)\tau,k\tau)$. So, since $u^k_\tau$ solves the Rothe problem, we have for almost every $t\in
(0,T)$
$$
(u'_\tau(t),v)_H + \langle A \bar{u}_\tau(t), v\rangle + \langle \xi_\tau(t), \iota
v\rangle_{U^*\times U} = \langle\bar{f}_\tau(t),v\rangle\ \mbox{for}\ v\in V,
$$
with $u_\tau(0)=u_{0\tau}$ and $\xi_\tau(t)=\xi^k_\tau\in\partial J(\iota u^k_\tau)=\partial
J(\iota \bar{u}_\tau(t))$ for $t\in ((k-1)\tau,k\tau]$. Defining the Nemytskii operator
$\C{A}:\C{V}\to \C{V}^*$ as $(\C{A}v)(t)=A(v(t))$, we have
\begin{equation}
(u'_\tau,v)_{\C{H}} + \langle \C{A} \bar{u}_\tau, v\rangle_{\C{V}^*\times\C{V}} +  \langle
\xi_{\tau}, \bar{\iota} v\rangle_{\C{U}^*\times
\C{U}}=\langle\bar{f}_\tau,v\rangle_{\C{V}^*\times\C{V}}\ \mbox{for}\ v\in
\C{V}.\label{eq:nemytskii}
\end{equation}
\begin{lemma}\label{lem:boundsnemytskii}
Under assumptions $H(A), H(J), H_0, H(U)$ and $H_{aux}$ there exists $\tau_0>0$ such that for all
$\tau\in (0,\tau_0)$, the piecewise constant and piecewise linear interpolants built on the
solutions of the Rothe problem satisfy
\begin{eqnarray}
&&\|\bar{u}_{\tau}\|_{\C{V}}\leq \mbox{const},\label{eq:bigbound1}\\
&&\|\bar{u}_{\tau}\|_{L^\infty(0,T;H)}\leq \mbox{const},\\
&&\|u_{\tau}\|_{C(0,T;H)}\leq \mbox{const},\label{eq:bigbound4}\\
&&\|u_\tau\|_{\C{V}}\leq \mbox{const},\label{eq:bigbound5}\\
&&\|u'_\tau\|_{\C{V}^*}\leq \mbox{const}\label{eq:bigbound6},\\
&&\|\C{A}\bar{u}_\tau\|_{\C{V}^*}\leq \mbox{const}\label{eq:bigbound7},\\
&&\|\xi_\tau\|_{\C{U}^*}\leq \mbox{const},\label{eq:bigbound8}\\
&&\|\bar{u}_{\tau}\|_{BV^2(0,T;V^*)}\leq \mbox{const}\label{eq:bigboundvar}.
\end{eqnarray}
with the constants independent on $\tau$.
\end{lemma}
\begin{proof}
Estimates (\ref{eq:bigbound1})-(\ref{eq:bigbound4}) follow directly from Lemma \ref{lem:bounds},
since
\begin{center}
$\|\bar{u}_{\tau}\|^2_{\C{V}}=\tau \sum_{i=1}^{N}\|u^k_\tau\|^2$,
$\|\bar{u}_{\tau}\|_{L^\infty(0,T;H)}=\max_{k=1,\ldots,N}\|u^k_\tau\|_H$ and
$\|u_{\tau}\|_{C(0,T;H)}\leq \max_{k=0,\ldots,N}\|u^k_\tau\|_H$.
\end{center}
\noindent The simple calculation shows us that $\|u_\tau\|_{\C{V}}^2\leq
\tau\sum_{k=0}^N\|u^k_\tau\|_V^2$. This, together with the fact, that $\|u^0_\tau\|\leq
C/\sqrt{\tau}$, by Lemma \ref{lem:bounds} gives (\ref{eq:bigbound5}).

\noindent To prove (\ref{eq:bigbound6}) let us consider the inclusion (\ref{eq:nemytskii}). We have
\begin{eqnarray}
&&\|u'_\tau\|_{\C{V}^*}=\sup_{\|v\|_{\C{V}}\leq 1}\left|(u'_\tau,v)_{\C{H}}\right| = \nonumber\\
&&=\sup_{\|v\|_{\C{V}}\leq 1}\left|\langle\bar{f}_\tau,v\rangle_{\C{V}^*\times\C{V}} - \langle \C{A} \bar{u}_\tau, v\rangle_{\C{V}^*\times\C{V}} - \int_0^T \langle \xi_\tau(t), \iota v(t)\rangle_{U^*\times U}\ dt\right|\leq \nonumber\\
&&\leq \|\bar{f}_\tau\|_{\C{V}^*}+\sqrt{\int_0^T\|A \bar{u}_\tau(t)\|_{V^*}^2\
dt}+\|\iota\|_{\C{L}(V;U)}\sqrt{\int_0^T\|\xi_\tau(t)\|^2_{U^*}\ dt}\leq \nonumber
\\&&\leq\|\bar{f}_\tau\|_{\C{V}^*}+\sqrt{2a^2T+2b^2\|\bar{u}_{\tau}\|^2_{\C{V}}}+
\|\iota\|_{\C{L}(V;U)}\sqrt{2c^2T+2c^2\|\iota\|_{\C{L}(V;U)}^2\|\bar{u}_{\tau}\|^2_{\C{V}}}.\label{eq:derivative}
\end{eqnarray}
Desired bound is obtained by (\ref{eq:bigbound1}). Estimates that appear in
(\ref{eq:derivative}) prove also (\ref{eq:bigbound7}) and (\ref{eq:bigbound8}). It remains to prove (\ref{eq:bigboundvar}).
Let us assume that the seminorm $BV^2(0,T;V^*)$ of piecewise constant function $\bar{u}_\tau$ is realized by some division $0=t_0<t_1<\ldots< t_k=T$. Each $t_j$ is in some interval $((m_j-1)\tau, m_j\tau]$, so $\bar{u}_\tau(t_j) = u^{m_j}_\tau$ with $m_0=0$ and $m_k=N$ and $m_{i+1}>m_i$ for $i=1,\ldots,N-1$. Thus
$$\|\bar{u}_{\tau}\|^2_{BV^2(0,T;V^*)}=\sum_{j=1}^{k}\|u^{m_j}_\tau-u^{m_{j-1}}_\tau\|^2_{V^*}.$$
We use the inequality
$$
\|u^{m_j}_\tau-u^{m_{j-1}}_\tau\|^2_{V^*} \leq (m_j-m_{j-1}) \sum_{i=m_{j-1}+1}^{m_j} \|u^{i}_\tau-u^{i-1}_\tau\|^2_{V^*}.
$$
Thus
\begin{eqnarray}
&&\|\bar{u}_{\tau}\|^2_{BV^2(0,T;V^*)}\leq \sum_{j=1}^{k} \left((m_j-m_{j-1}) \sum_{i=m_{j-1}+1}^{m_j} \|u^{i}_\tau-u^{i-1}_\tau\|^2_{V^*}\right)\leq\nonumber\\
&& \leq \left( \sum_{j=1}^{k} (m_j-m_{j-1}-1) \right)\sum_{i=1}^N \|u^{i}_\tau-u^{i-1}_\tau\|^2_{V^*} \leq N \tau\tau \sum_{i=1}^N \left\|\frac{u^{i}_\tau-u^{i-1}_\tau}{\tau}\right\|^2_{V^*} =\nonumber \\
&& = T \int_0^T\|u_\tau'(t)\|_{V^*}^2\ dt.\nonumber
\end{eqnarray}
The last term is bounded by (\ref{eq:bigbound6}), which ends the proof.
\end{proof}

\begin{thm}\label{thm:existence}
Under assumptions $H(A), H(J), H_0, H(U)$ and $H_{aux}$ the problem (\ref{eq:inclusionmain}) has a
solution $u$. Furthermore if $\bar{u}_\tau$ and $u_\tau$ are piecewise constant and piecewise
linear interpolants built on the solutions of the Rothe problem, then, for a subsequence,
$u_\tau\to u$ weakly in $\C{W}$ and weakly$*$ in $L^\infty(0,T;H)$ and $\bar{u}_\tau\to u$ weakly
in $\C{V}$ and weakly$*$ in $L^\infty(0,T;H)$.
\end{thm}
\begin{proof}
From the bounds obtained in Lemma \ref{lem:boundsnemytskii}, possibly for a subsequence, we get
\begin{eqnarray}
&&\bar{u}_{\tau}\to u\ \mbox{weakly in}\ \C{V}\ \mbox{and weakly$*$ in}\ L^\infty(0,T;H),\label{eq:convuvstar}\\
&&u_{\tau}\to u_1\ \mbox{weakly in}\ \C{V} \ \mbox{and weakly$*$ in}\ L^\infty(0,T;H),\\
&&u'_{\tau}\to u_2\ \mbox{weakly in}\ \C{V}^*,\label{eq:convuprimevstar}\\
&&\C{A}\bar{u}_{\tau}\to \eta \ \mbox{weakly in}\ \C{V}^*,\\
&&\xi_\tau\to \xi \ \mbox{weakly in}\ \C{U}^*.\label{eq:convxiustar}
\end{eqnarray}

\noindent A standard argument shows that $u_1'=u_2$. To show that $u=u_1$ we observe that
\begin{equation}
\|\bar{u}_{\tau}-u_\tau\|_{\C{V}^*}^2=\sum_{k=1}^N\int_{(k-1)\tau}^{k\tau}\left(k\tau-t\right)^2\left\|\frac{u^k_\tau-u^{k-1}_\tau}{\tau}\right\|_{V^*}^2\
dt=\frac{\tau^2}{3}\|u'_\tau\|_{\C{V^*}}^2,
\end{equation}
which means that $\bar{u}_\tau-u_{\tau} \to 0$ strongly in $\C{V}^*$ as $\tau\to 0$, and, in
con\-se\-quence $u=u_1$.

\noindent It follows that $u_\tau\to u$ strongly in $L^2(0,T;H)$ and weakly in $C([0,T];H)$. This
also implies that $u_{0\tau}=u_\tau(0)\to u(0)$ weakly in $H$, so $u(0)=u_0$.

\noindent A passage to the limit in (\ref{eq:nemytskii}) gives
$$
u'+\eta+\bar{\iota}^*\xi=f.
$$
We observe that, by $H(U)$, we have $\bar{\iota}\bar{u}_{\tau}\to \bar{\iota}u$ strongly in $\C{U}$
and, furthermore, for a subsequence $\iota \bar{u}_{\tau}(t)\to \iota u(t)$ strongly in $U$ for a.e. $t\in (0,T)$.
Moreover $\xi_\tau\to\xi$ weakly in $L^1(0,T;U^*)$. Since $\partial J:U\to 2^{U^*}$ has nonempty,
closed and convex values and is upper semicontinuous from $U$ furnished with strong topology into
$U^*$ furnished with weak topology (see \cite{Denkowski2003a}, Proposition 5.6.10), by the Convergence
Theorem of Aubin and Cellina (see \cite{Aubin1984}, Theorem 1, Section 1.4), we deduce that
$\xi(t)\in \partial J(\iota u(t))$ for a.e. $t\in (0,T)$. In order to show that $u$ satisfies the
inclusion (\ref{eq:inclusionmain}), it suffices to prove that $\eta=\C{A}u$. To this end, let us
estimate
\begin{eqnarray}
&&\limsup_{\tau\to 0}\langle\C{A}\bar{u}_\tau,\bar{u}_\tau-u\rangle_{\C{V}^*\times\C{V}}\leq
\limsup_{\tau\to 0}\langle\bar{f}_\tau,\bar{u}_\tau-u\rangle_{\C{V}^*\times\C{V}} -\\
&&-\liminf_{\tau\to
0}\langle u_\tau ' ,\bar{u}_\tau-u\rangle_{\C{V}^*\times\C{V}}-\liminf_{\tau\to 0}\langle \xi_\tau,
\bar{\iota}(\bar{u}_\tau-u)\rangle_{\C{U}^*\times\C{U}}.\nonumber
\end{eqnarray}
Since $\bar{f}_\tau\to f$ strongly in $\C{V}^*$, by (\ref{eq:convuvstar}), we get $\lim_{\tau\to
0}\langle\bar{f}_\tau,\bar{u}_\tau-u\rangle_{\C{V}^*\times\C{V}}=0.$ Moreover, since
$\bar{\iota}\bar{u}_\tau\to \bar{\iota}u$ strongly in $\C{U}$, by (\ref{eq:convxiustar}), we have
$\lim_{\tau\to 0}\langle \xi_\tau, \bar{\iota}(\bar{u}_\tau-u)\rangle_{\C{U}^*\times\C{U}}=0$. Now
we observe that
\begin{eqnarray}
&&\langle u_\tau ' ,\bar{u}_\tau-u\rangle_{\C{V}^*\times\C{V}} = \langle u_\tau '
,\bar{u}_\tau-u_\tau\rangle_{\C{V}^*\times\C{V}} +\nonumber\\
&&+\frac{1}{2}(\|u_\tau(T)-u(T)\|_H^2-\|u_\tau(0)-u(0)\|_H^2) + \langle u', u_\tau -
u\rangle_{\C{V}^*\times\C{V}},\nonumber
\end{eqnarray}
so, noting that $\langle u_\tau ' ,\bar{u}_\tau-u_\tau\rangle_{\C{V}^*\times\C{V}} \geq 0$, we
obtain
$$
\liminf_{\tau\to 0}\langle u_\tau ' ,\bar{u}_\tau-u\rangle_{\C{V}^*\times\C{V}} \geq 0.
$$
Thus we have
$$
\limsup_{\tau\to 0}\langle\C{A}\bar{u}_\tau,\bar{u}_\tau-u\rangle_{\C{V}^*\times\C{V}}\leq 0.
$$
We are in a position to apply Lemma \ref{lem:nemytskii} which gives $\eta=\C{A}u$. Thus $u$
solves (\ref{eq:inclusionmain}).
\end{proof}

\noindent \textbf{Remark 5.} Note that we have also proved that any cluster point of $u_\tau$ and
$\bar{u}_\tau$, in the sense (\ref{eq:convuvstar})-(\ref{eq:convuprimevstar}), solves the problem
(\ref{eq:inclusionmain}). It is not known, however, whether there are solutions which are not
limits of the interpolants built on the solutions of Rothe problem.

\section{Uniqueness and strong convergence}\label{sec:uniq}
\noindent In this section we assume the strong monotonicity type relation for $A$ and relaxed
monotonicity on $J$.
\begin{itemize}
\item[$H(A)_1$]: assumptions $H(A)$ hold and $A$ satisfies the monotonicity type relation $\langle Au-Av,u-v\rangle\geq m_1 \|u-v\|^2-m_2\|u-v\|_H^2$ for every $u, v\in V$ with $m_1 \geq 0$ and $m_2 > 0$,
\item[$H(A)_2$]: assumptions $H(A)$ hold and the Nemytskii mapping $\C{A}:\C{V}\to\C{V}^*$ is of class $(S_+)$ with respect to the space $M^{2,2}(0,T;V,V^*)$, that is if $u_n\to u$ weakly in $\C{V}$ and $u_n$ is bounded in $M^{2,2}(0,T;V,V^*)$ then $\limsup_{n\to\infty}\langle\C{A} u_n,u_n-u\rangle_{\C{V}^*\times\C{V}}\leq 0$ implies that $u_n\to u$ strongly in $\C{V}$,
\item[$H(J)_1$]: assumptions $H(J)$ hold and $J$ satisfies the relaxed monotonicity condition $\langle \xi - \eta, u - v\rangle_{U\times U^*} \geq - m_3\|u-v\|_U^2$ for every $u, v\in V$ and $\xi\in \partial J(u), \eta\in \partial J(v)$ with $m_3 > 0$,
\item[$H_{const}$]: either $H_{aux}$ A) holds or $m_1 \geq m_3 \|\iota\|_{\C{L}(V;U)}^2$.
\end{itemize}

\noindent\textbf{Remark 6.} The assumption $H(A)_1$ for the divergence differential
Leray-Lions operator is guaranteed by appropriate Leray-Lions type conditions. For $H(A)_2$ to hold it suffices that the operator $A$ is of class $(S_+)$, by an argument analogous to Theorem 2(c) in \cite{Berkovits1996}.

\noindent\textbf{Remark 7.} The relaxed monotonicity condition $H(J)_1$ (which is associated with
the semiconvexity of the functional $J$) was already used to prove the uniqueness of solutions to
the first order evolution parabolic hemivariational inequalities in \cite{Liu2005} and second order
ones in \cite{Migorski2005}.

\noindent\textbf{Remark 8.} Note that $H(A)_1$ allows the case $m_1=0$, but if the inequality in
$H_{const}$ holds, then it must be $m_1>0$.

\begin{thm}\label{thm:unique}
Under assumptions $H(A)_1$, $H(J)_1$, $H_0$, $H(U)$, $H_{aux}$ and $H_{const}$, the solution to the
problem (\ref{eq:inclusionmain}) is unique.
\end{thm}
\begin{proof}
Assume that $u_1, u_2$ are two distinct solutions to the problem (\ref{eq:inclusionmain}). We have,
for $v\in V$ and a.e. $t\in [0,T]$
\begin{equation}
\langle (u_1-u_2)'(t), v\rangle + \langle Au_1(t) - Au_2(t), v \rangle +
\langle\xi(t)-\eta(t),\iota v \rangle_{U\times U^*} = 0,
\end{equation}
where $\xi(t)\in \partial J(\iota u_1(t))$ and $\eta(t)\in \partial J(\iota u_2(t))$ for a.e. $t\in
(0,T)$. Taking $v = u_1(t)-u_2(t)$, we obtain
\begin{eqnarray}
&&\frac{1}{2}\frac{d}{dt}\|u_1(t)-u_2(t)\|^2_H + \langle Au_1(t) - Au_2(t), u_1(t)- u_2(t)\rangle +\\
&&+\langle\xi(t)-\eta(t),\iota u_1(t)-\iota u_2(t)) \rangle_{U\times U^*} = 0.\nonumber
\end{eqnarray}
Application of $H(A)_1$ and $H(J)_1$ gives for a.e. $t\in (0,T)$
\begin{eqnarray}
&&\frac{1}{2}\frac{d}{dt}\|u_1(t)-u_2(t)\|^2_H + m_1 \|u_1(t) - u_2(t)\|^2 -\\
&&-m_2\|u_1(t)-u_2(t)\|_H^2-m_3 \|\iota(u_1(t)-u_2(t))\|_U^2\leq  0.\nonumber
\end{eqnarray}
By $H_{const}$ we have either
\begin{equation}
\frac{1}{2}\frac{d}{dt}\|u_1(t)-u_2(t)\|^2_H \leq  m_2\|u_1(t)-u_2(t)\|_H^2,
\end{equation}
or, in the case of $H_{aux}$ A),
\begin{equation}
\frac{1}{2}\frac{d}{dt}\|u_1(t)-u_2(t)\|^2_H \leq  (m_2+\|p\|_{\C{L}(H;U)})\|u_1(t)-u_2(t)\|_H^2,
\end{equation}
which, by the Gronwall lemma, gives the thesis.
\end{proof}
\noindent \textbf{Remark 9.} Under assumptions of Theorem \ref{thm:unique}, the convergences in
Theorem \ref{thm:existence} hold for the whole sequences $u_\tau$ and $\bar{u}_\tau$.
\begin{thm}\label{thm:converge}
Let assumptions $H(A)_1$, $H(J)$, $H_0$, $H(U)$, $H_{aux}$ hold and the subsequences $u_\tau,
\bar{u}_\tau$ converge in the sense (\ref{eq:convuvstar})-(\ref{eq:convuprimevstar}). Then
$u_\tau\to u$ strongly in $C([0,T];H)$. If instead of $H(A)_1$ we assume $H(A)_2$, then $\bar{u}_\tau\to u$
strongly in $\C{V}$.
\end{thm}
\begin{proof}
Let $u_\tau$ and $\bar{u}_\tau$ be interpolants built on the solutions of the Rothe problem and let
$u$ be the solution to (\ref{eq:inclusionmain}) obtained in Theorem {\ref{thm:existence}}. For
$v\in V$ and a.e. $t\in (0,T)$ we get
\begin{equation}
\langle u_\tau '(t) - u'(t),v\rangle + \langle A\bar{u}_\tau (t) -
Au(t),v\rangle+\langle\xi_\tau(t)-\eta(t),\iota v\rangle_{U^*\times U}=\langle
\bar{f}_\tau(t)-f(t),v\rangle,
\end{equation}
where $\xi_\tau(t) \in \partial J(\iota \bar{u}_\tau(t))$ and $\eta(t) \in \partial J(\iota u(t))$
for a.e. $t\in (0,T)$. Choosing $v=\bar{u}_\tau(t)-u(t)$, we get
\begin{eqnarray}
&&\langle u_\tau '(t) - u'(t),\bar{u}_\tau (t) -u_\tau (t)\rangle + \frac{1}{2}\frac{d}{dt}\|u_\tau (t)- u(t)\|_H^2 +\nonumber\\
&& +\langle A\bar{u}_\tau (t) - Au(t),\bar{u}_\tau (t) - u(t)\rangle\leq \nonumber\\
&&\leq\langle \bar{f}_\tau(t)-f(t),\bar{u}_\tau(t)-u(t)\rangle + \|\iota \bar{u}_\tau(t)-\iota
u(t)\|_U\|\xi_\tau(t)-\eta(t)\|_{U^*}.\nonumber
\end{eqnarray}
Since $\langle u_\tau '(t), \bar{u}_\tau (t) -u_\tau (t)\rangle = \|\frac{d}{dt}u_\tau
'(t)\|_H^2(k\tau-t) \geq 0$ for any $t\in ((k-1)\tau, k\tau)$ for a.e. $t\in (0,T)$ we have
\begin{eqnarray}\label{eqn:integration}
&&\frac{1}{2}\frac{d}{dt}\|u_\tau (t)- u(t)\|_H^2 + \langle A\bar{u}_\tau (t) - Au(t),\bar{u}_\tau (t) - u(t)\rangle \leq \nonumber\\
&&\leq \|\iota \bar{u}_\tau(t)-\iota u(t)\|_U\|\xi_\tau(t)-\eta(t)\|_{U^*}+\nonumber\\
&&+ \langle \bar{f}_\tau(t)-f(t),\bar{u}_\tau(t)-u(t)\rangle+\langle u'(t),\bar{u}_\tau (t) -u_\tau
(t)\rangle.
\end{eqnarray}
Using $H(A)_1$ and integrating the last inequality, for $t\in[0,T]$, we get
\begin{eqnarray}
&&\frac{1}{2}\|u_\tau (t)- u(t)\|_H^2 \leq m_2\int_0^t\|u_\tau (s)- u(s)\|_{H}^2 \ dt + \nonumber\\
&&+\sqrt{2}c\|\bar{\iota} (\bar{u}_\tau-u)\|_{\C{U}}(2T+\|\bar{\iota}\bar{u}_\tau\|_{\C{U}}+\|\bar{\iota} u\|_{\C{U}})+\nonumber\\
&&+ \|\bar{f}_\tau-f\|_{\C{V}^*}(\|\bar{u}_\tau\|_{\C{V}}+\|u\|_{\C{V}})+\langle u',\bar{u}_\tau
-u_\tau\rangle_{\C{V}^*\times \C{V}}+\frac{1}{2}\|u_{0 \tau}- u_0\|_H^2.\nonumber
\end{eqnarray}
The Gronwall lemma gives the strong convergence of $u_\tau$ to $u$ in $C([0,T];H)$.

\noindent In order to obtain the strong convergence in $\C{V}$, let us integrate
(\ref{eqn:integration}) over $(0,T)$. We have
\begin{eqnarray}
&&\frac{1}{2}\|u_\tau (T)- u(T)\|_H^2 + \langle\C{A}\bar{u}_\tau-\C{A}u,\bar{u}_\tau-u\rangle_{\C{V}^*\times\C{V}}\leq \nonumber\\
&&\leq \sqrt{2}c\|\bar{\iota} (\bar{u}_\tau-u)\|_{\C{U}}(2T+\|\bar{\iota}\bar{u}_\tau\|_{\C{U}}+\|\bar{\iota} u\|_{\C{U}})+\nonumber\\
&&+ \|\bar{f}_\tau-f\|_{\C{V}^*}(\|\bar{u}_\tau\|_{\C{V}}+\|u\|_{\C{V}})+\langle u',\bar{u}_\tau
-u_\tau\rangle_{\C{V}^*\times \C{V}}+\frac{1}{2}\|u_{0 \tau}- u_0\|_H^2.\nonumber
\end{eqnarray}
Passing to the limit, we get
$$
\limsup_{\tau\to 0}\langle \C{A}\bar{u}_\tau-\C{A}u,\bar{u}_\tau-u\rangle_{\C{V}^*\times\C{V}} \leq
0.
$$
The thesis is implied by $H(A)_2$.
\end{proof}
\noindent \textbf{Remark 10.} If, in addition to assumptions of the Theorem \ref{thm:converge},
also $H(J)_1$ and $H_{const}$ hold, then, by Theorem \ref{thm:unique}, the whole sequences $u_\tau$
and $\bar{u}_\tau$ converge strongly in $C([0,T];H)$ and $\C{V}$ respectively.
\section{Examples}\label{sec:settings}
\noindent In this section we provide examples of that problem setup which are particular case of
the general problem considered previously. Moreover, we present a simple numerical example.

\noindent\textit{Problem settings} We assume that $\Omega\subset \mathbb{R}^n$  is an open and
bounded domain with smooth boundary. The space $V$ is either $H^1(\Omega;\mathbb{R}^m)$ with
$m\in\mathbb{N}$ (possibly, but not necessarily, $m=n$) or its closed subspace (which originates
from homogeneous Dirichlet boundary condition on $\Gamma_D\subset \partial \Omega$). Furthermore
let $H=L^2(\Omega;\mathbb{R}^m)$. Then the embedding $i:V\to H$ is continuous and compact. We
consider two examples.
\begin{itemize}
\item Multivalued term is defined on $\Omega$. We specify $\Lambda \subset \Omega$ to be an open subset on nonzero measure and fix $d\in \mathbb{N}$.
Furthermore we assume that $\C{M}\in L^\infty(\Lambda;\C{L}(\mathbb{R}^m;\mathbb{R}^d))$. Now
$U=L^2(\Lambda;\mathbb{R}^d)$. The mapping $\iota$ is defined by $(\iota
v)(x)=\C{M}(x)((iv)|_{\Lambda}(x))$. We observe that $\iota:V\to U$ is linear, continuous and
compact. By Proposition \ref{prop:compactembedding}, the embedding $M^{2,2}(0,T;V,V^*)\subset
L^2(0,T;H)$ is compact, which implies $H(U)$. Defining $p:H\to U$ by
$(pv)(x)=\C{M}(x)(v|_{\Lambda}(x))$, we see that $A)$ of $H_{aux}$ is satisfied. The solution
exists under assumptions $H(A), H(J)$ and $H_0$ (Theorem \ref{thm:existence}). Additional
assumptions $H(A)_1$ and $H(J)_1$ imply uniqueness of solution by Theorem \ref{thm:unique} and
strong convergence of the Rothe sequence $u_\tau$ in $C([0,T];H)$. Furthermore, if $H(A)_2$ holds,
then the sequence $\bar{u}_\tau$ converges strongly in $\C{V}$.

\noindent As the special case we can consider $\Lambda=\Omega$, $m=n=d$ and $\C{M}(x)\equiv I$
(identity) for all $x\in \Omega$. Then we recover $U=H$, which gives existence results in spirit of
\cite{Migorski2001}.

\item Multivalued term is defined on the boundary of $\Omega$. We specify $\Gamma_C \subset \partial \Omega$
disjoint with $\Gamma_D$. We take $Z=H^{\delta}(\Omega;\mathbb{R}^m)$ with
$\delta\in[\frac{1}{2},1)$. The continuous and compact embedding $V\to Z$ is denoted by $\bar{i}$
and the trace operator is given by $\bar{\gamma}:Z\to L^2(\Gamma_C;\mathbb{R}^m)$. Furthermore let
$d\in \mathbb{N}$ and $\C{M}\in L^\infty(\Gamma_C;\C{L}(\mathbb{R}^m;\mathbb{R}^d))$. Now
$U=L^2(\Gamma_C;\mathbb{R}^d)$. The mapping $\iota$ is defined by $(\iota
v)(x)=\C{M}(x)((\bar{\gamma}\bar{i}v)(x))$. The mapping $\iota:V\to U$ is linear, continuous and
compact. The spaces $V\subset Z\subset V^*$ satisfy the assumptions of Proposition
\ref{prop:compactembedding}, so $M^{2,2}(0,T;V,V^*)$ is embedded in $L^2(0,T;Z)$ compactly. Therefore
the assumption $H(U)$ is satisfied. Since claim $A)$ of $H_{aux}$ does not hold in this case, in
order to obtain the existence of solutions (Theorem \ref{thm:existence}) we need to assume $H(A),
H(J), H_0$ and either $B)$ or  $C)$ of $H_{aux}$. Furthermore, if $H(A)_1$ and $H(J)_1$ hold, then
a subsequence of the Rothe sequence $u_\tau$ converges strongly in $C([0,T];H)$ (Theorem
\ref{thm:converge}). If moreover $H(A)_2$ holds, then we also have the strong convergence of
$\bar{u}_\tau$ in $\C{V}$. If furthermore the relation between $m_1$ and $m_3$ given by $H_{const}$
holds, then the solution is unique (Theorem \ref{thm:unique}) and the whole Rothe sequences
$u_\tau$ and $\bar{u}_\tau$ converge strongly in $C([0,T];H)$ and $\C{V}$ respectively.

\noindent In the case $m=d=1$ and $\C{M}(x)\equiv I$ we recover the results of \cite{Migorski2004}.
If $m=n>1$ and $\nu$ is the unit outer normal versor on the boundary $\partial \Omega$, then two
special cases are $d=1$, $\C{M}(x)(a)=\nu(x)\cdot a$ and $d=m$, $\C{M}(x)(a)=a-(\nu(x)\cdot
a)\nu(x)$. We recover the cases of the boundary conditions given in normal and tangent directions,
respectively.
\end{itemize}

\noindent \textit{Numerical example}. Let us take $\Omega=(0,1)$. The problem under consideration
will be
\begin{eqnarray}
&&u_t(x,t)=u_{xx}(x,t)\ \ \mbox{for}\ \ (x,t)\in \Omega\times(0,T),\\
&&u(0,t)=0\ \ \mbox{for}\ \ t\in(0,T),\\
&&u_x(1,t)\in -\partial j(u(1,t))\ \ \mbox{for}\ \ t\in(0,T),\label{eq:boundex}\\
&&u(x,0)=u_0(x)\ \ \mbox{for}\ \ x\in(0,1).
\end{eqnarray}
We set $V=\{v\in H^1(0,1): v(0)=0\}$ and $H=L^2(0,1)$. Taking $t_k=k\Delta t$, $u(x,t_k):=u^k(x)$
and $v\in V$, the Rothe problem has the following form
\begin{eqnarray}
&&\int_0^1\frac{u^{k+1}(x)-u^k(x)}{\Delta t}v(x)\ dx+\int_0^1u^{k+1}_x(x)v_x(x)\ dx +
\partial j(u^{k+1}(1))v(1)\ni 0,\nonumber\\
&&u^0(x)=u_0(x).
\end{eqnarray}
The problem in each time step will be solved by the Galerkin scheme. Let $V_n$ be a subspace of $V$
consisting of piecewise linear functions constructed on a uniform mesh $x_0=0, \ldots,x_i=i\Delta
x,\ldots, x_n=n\Delta x=1$ in $(0,1)$ such that $\mbox{dim} V_n = n$. Let furthermore $u^k(x)$ be
approximated by $\sum_{i=1}^{n}\alpha_i^k v_i$, where $\{v_i\}_{i=1}^{n}$ forms the base of $V_n$ given by the duality condition $v_j(x_i)=\delta_{ij}$.
We assume that $\alpha_n^k$ is the value of the solution at the last mesh point ($x=1$). The values
$\alpha_i^{k+1}$ satisfy, for $j=1,\ldots, n$
\begin{eqnarray}
&&\sum_{i=1}^n\alpha_i^{k+1}\left(\frac{1}{\Delta t}\int_0^1v_i(x)v_j(x)\ dx+\int_0^1 v_i'(x)v_j'(x)\
dx\right)-\nonumber\\
&&-\sum_{i=1}^n\alpha_i^k\frac{1}{\Delta t}\int_0^1v_i(x)v_j(x)\ dx + \partial
j(\alpha_n^{k+1}) v_j(1)\ni 0.\nonumber
\end{eqnarray}
Calculating the integrals and denoting $d=\frac{\Delta t}{\Delta x^2}$, for $j=2,\ldots,n-1$, we
obtain
\begin{equation}\label{eqn:j_1_to_k_m_1}
\alpha_{j-1}^{k+1}\left(\frac{1}{6}-d\right)+\alpha_j^{k+1}\left(\frac{2}{3}+2d\right)+
\alpha_{j+1}^{k+1}\left(\frac{1}{6}-d\right)=\alpha_{j-1}^k\frac{1}{6}+\alpha_j^k\frac{2}{3}+
\alpha_{j+1}^k\frac{1}{6},
\end{equation}
and for $j=n$, we have
\begin{equation}\label{eqn:j_k}
\alpha_{n-1}^{k+1}\left(\frac{1}{6}-d\right)+\alpha_n^{k+1}\left(\frac{1}{3}+d\right)+ \frac{\Delta
t}{\Delta x}\partial j(\alpha_n^{k+1})\ni\alpha_{n-1}^k\frac{1}{6}+\alpha_n^k\frac{1}{3}.
\end{equation}
Finally for the left, Dirichlet, boundary point we have 
\begin{equation}\label{eqn:j_1}
\alpha_1^{k+1}\left(\frac{2}{3}+2d\right)+
\alpha_{2}^{k+1}\left(\frac{1}{6}-d\right)=\alpha_1^k\frac{2}{3}+
\alpha_{2}^k\frac{1}{6}.
\end{equation}
 \begin{figure}[h]
        \centering
        \includegraphics[width=\textwidth]{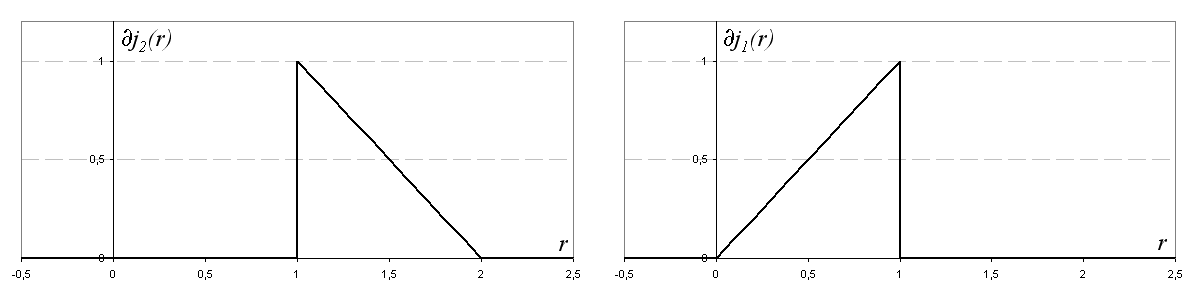}
        \caption{\footnotesize{Examples of multifunctions used as nonmonotone and multivalued boundary conditions (\ref{eq:boundex})}}
        \label{fig:examples}
 \end{figure}
We consider two examples of the locally Lipschitz functionals $j$:
$$
j_1(r)=\begin{cases} 0\  \mbox{for}\ r\leq 0 \\ \frac{r^2}{2}\  \mbox{for}\ r\in(0,1)\\
\frac{1}{2}\  \mbox{for}\ r\geq 1\end{cases}\ \ j_2(r)=\begin{cases} 0\  \mbox{for}\ r\leq 1 \\
\frac{1-(r-2)^2}{2}\  \mbox{for}\ r\in(1,2)\\ \frac{1}{2}\  \mbox{for}\ r\geq 2.\end{cases}
$$
Their subdifferentials in the sense of Clarke are given by
\begin{equation}\label{eqn:subdiff_egzamples}
\partial j_1(r)=\begin{cases} 0\  \mbox{for}\ r\leq 0\ \mbox{or}\ r>1 \\ r\  \mbox{for}\ r\in(0,1)\\ [0,1]\  \mbox{for}\ r=1\end{cases}\ \ \partial j_2(r)=\begin{cases} 0\  \mbox{for}\ r<1\ \mbox{or}\ r\geq 2 \\ 2-r\  \mbox{for}\ r\in(1,2)\\ [0,1]\  \mbox{for}\ r=1.\end{cases}
\end{equation}
 \begin{figure}[h]
        \centering
         \includegraphics[width=\textwidth]{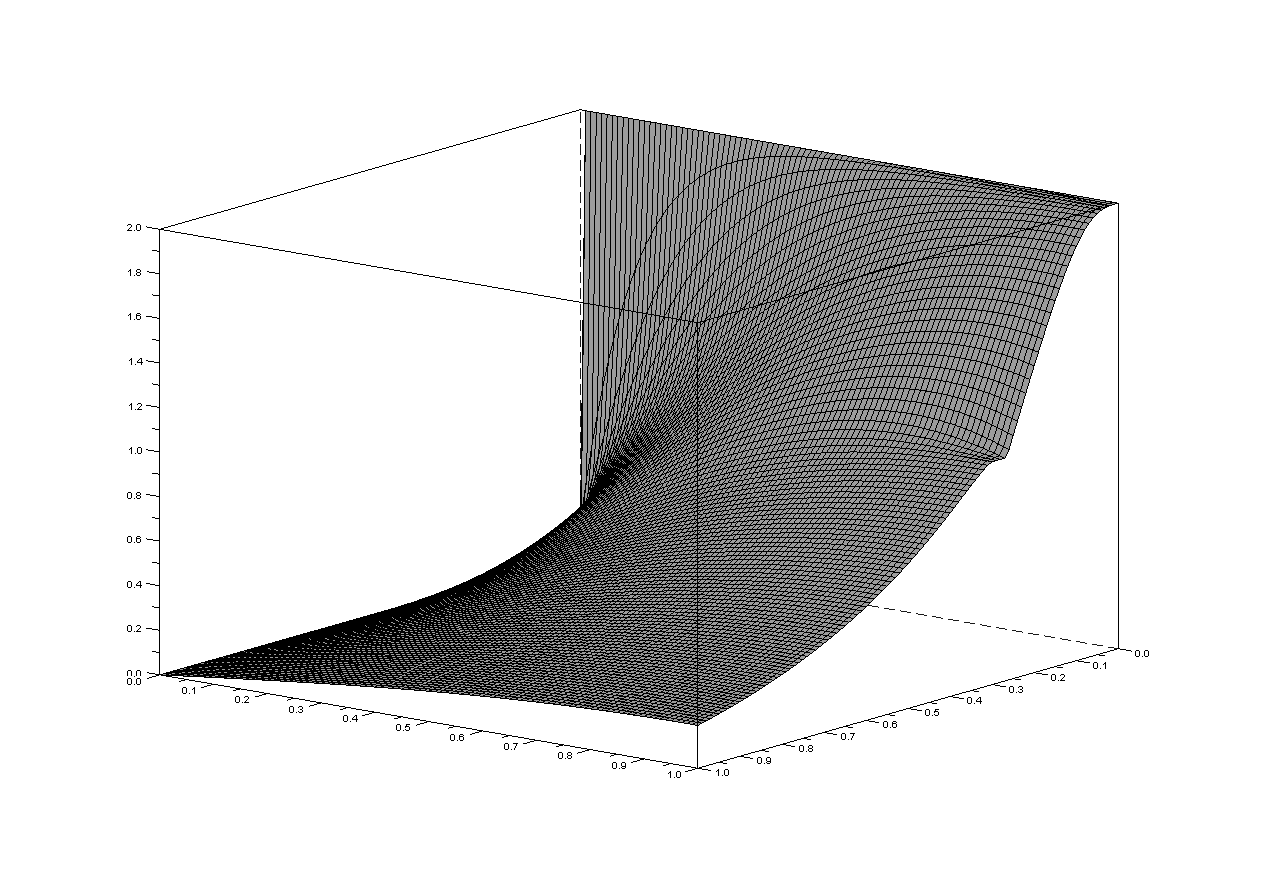}
        \caption{\footnotesize{Simulation for the potential $j_2$. Plot of unique obtained solution is drawn.}}
        \label{fig:j2}
 \end{figure}
 \begin{figure}[h]
        \centering
        \includegraphics[width=\textwidth]{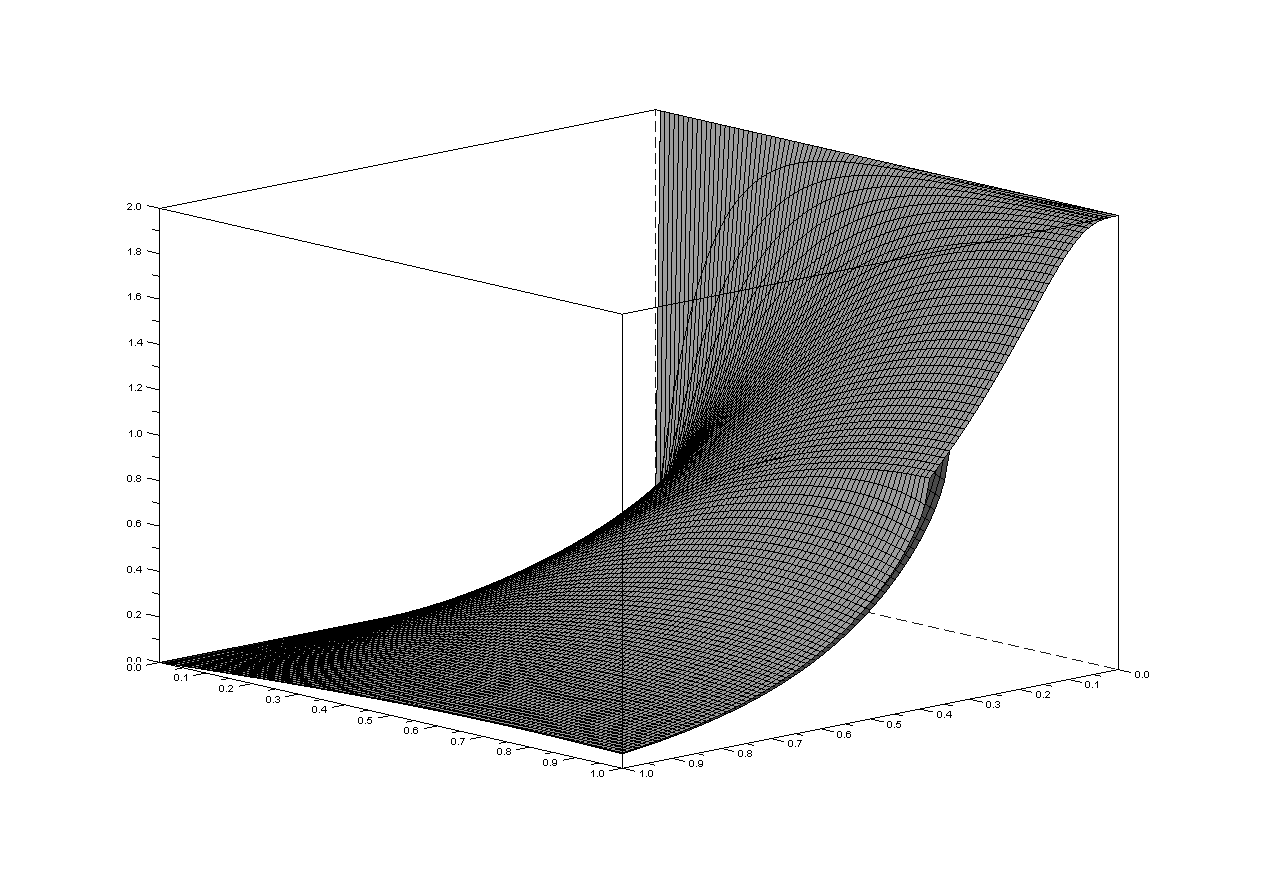}
        \caption{\footnotesize{Simulation for the potential $j_1$. Plots of two (respectively maximal and minimal one) of obtained many solutions are drawn.}}
        \label{fig:j1}
 \end{figure}
\noindent The graphs of $\partial j_1$ and $\partial j_2$ are presented in Figure
\ref{fig:examples}. Both potentials satisfy $H(J)$. The potential $j_1$ does not satisfy $H(J)_1$
since its subdifferential has a nonmonotone jump. The potential $j_2$ satisfies $H(J)_1$ since its
subdifferential has a monotone jump and nonmonotonicity is Lipschitz. In the case of $j_1$ the
question of multiplicity of solutions remains an open problem (however the numerical simulation
below show that it is more likely that there are multiple solutions) and in the case of $j_2$, a
single solution is expected (at least as long as the inequality in $H_{const}$ holds).

\noindent In both cases we take $u_0(x)\equiv 2$. The following scheme is used to find solutions of
(\ref{eqn:j_1_to_k_m_1})-(\ref{eqn:j_1}). In every time step at most three solutions can be found:
\begin{itemize}
\item Assume that the element of $\partial j(\alpha_{n}^{k+1})$ for which (\ref{eqn:j_k}) holds is equal to $0$ i.e.
we fall on the horizontal line in the graph of $\partial j$. Solve the system of $n$ equations
(\ref{eqn:j_1_to_k_m_1})-(\ref{eqn:j_1}) and verify whether obtained $\alpha_n^{k+1}$ falls in the
corresponding interval.
\item Assume that the element of $\partial j(\alpha_{n}^{k+1})$ for which (\ref{eqn:j_k}) holds is on the oblique
line in the graph of $\partial j$. Solve the system of $n$ equations
(\ref{eqn:j_1_to_k_m_1})-(\ref{eqn:j_1}) and verify whether obtained $\alpha_n^{k+1}$ is in the
corresponding interval.
\item Assume that we fall on the vertical line in the graph of $\partial j$, i.e. $\alpha_n^{k+1}=1$. Solve the system of $n-1$ equations (\ref{eqn:j_1_to_k_m_1}) and (\ref{eqn:j_1}). Then verify if $\partial j(\alpha_n^{k+1})$ calculated from (\ref{eqn:j_k}) falls in the interval $[0,1]$.
\end{itemize}
The simulations were run for $\Delta t = 0.01$ and $\Delta x = 0.01$. For the case of $j_2$ only
one numerical solution was obtained (i.e. in every time step only one of above three cases
occurred). The result is presented in Figure \ref{fig:j2}. For the case of $j_1$ many solutions
were obtained (i.e. there were time steps in which more then one of above three cases occurred).
Figure \ref{fig:j1} shows two solutions with respectively maximum and minimum value of
$\alpha_n^{k}$ chosen in each time step in which multiple solutions were found.

\end{document}